\documentclass[12pt]{amsart}
\usepackage{amssymb}
\input amssym.def
\usepackage{amsmath,amsfonts,hyperref,xcolor,mathtools}
\usepackage{amscd}
\usepackage[mathscr]{eucal}
\usepackage{enumitem}
\usepackage[thinc]{esdiff}
\usepackage{orcidlink}
\usepackage{float}
\usepackage{bm}

\hypersetup{colorlinks=true,citecolor={purple},linkcolor={teal},urlcolor={violet}}

\newcommand{\N}{{\mathbb N}}

\newcommand{\Q}{{\mathbb Q}}
\newcommand{\C}{{\mathbb C}}
\newcommand{\R}{{\mathbb R}}

\newtheorem{thm}{Theorem}

\newtheorem{cor}{Corollary}

\newtheorem{rmk}{Remark}

\newcommand{\thmref}[1]{Theorem~\ref{#1}}

\newcommand{\rmkref}[1]{Remark~\ref{#1}}

\makeatletter
\@namedef{subjclassname@2020}{%
\textup{2020} Mathematics Subject Classification}
\makeatother

\begin{document}


\title[On the multiple zeta functions and their variants at identical arguments]
{A note on the multiple zeta functions and their variants at identical arguments}

\author{Pawan Singh Mehta \orcidlink{0009-0001-6717-9063}}

\address{Pawan Singh Mehta\\ \newline
Department of Mathematics, Indian Institute of Technology Delhi, 
Hauz Khas, New Delhi 110016, India.}
\email{Pawan.Singh.Mehta@maths.iitd.ac.in}



\begin{abstract}
In this article, we study the multiple zeta functions (MZF) and some of its variants at identical arguments. Using the harmonic product, these functions can be expressed as polynomials in the Riemann zeta function. Firstly, we note that an explicit description of the coefficients of these polynomials can be given in terms of the values of the complete Bell polynomials. This for instance, easily leads to the complete description of the singularities of these functions. But more importantly, this enables us to establish a functional relation between MZF, Multiple $t$-functions (M$t$F) and their star variants at identical arguments.
\end{abstract}

\maketitle
\section{Introduction}
Throughout this article, we denote the set of non-negative integers by $\N$.
For an integer $r \ge 1$, the depth $r$ multiple zeta function is defined by
$$
\zeta_r(s_1,\ldots,s_r):=\sum_{n_1>\cdots>n_r>0}\frac{1}{n_1^{s_1}\cdots n_r^{s_r}},
$$
on the open set
$$
U_r:=\{(s_1,\ldots,s_r)\in \C^r : \Re(s_1+\cdots+s_i)>i \text{ for all } 1\leq i\leq r\}.
$$
We also set $\zeta_0(\varnothing):=1$, in the depth $0$ case. In depth $1$, we get back the classical Riemann zeta function $\zeta(s)$. In the last three decades, the multiple zeta functions and their special values have been studied extensively. However, their origin can be traced back to Euler. While studying the special values of the Riemann zeta function, he noticed that
\begin{align}\label{eq-Euler}
\zeta(s_1)\zeta(s_2)=\zeta_2(s_1,s_2)+\zeta_2(s_2,s_1)+\zeta(s_1+s_2),
\end{align}
for $s_1,s_2\in \C$, with $\Re(s_1),\Re(s_2)>1$.
This product formula is known as the harmonic product formula (or, the stuffle product formula) and can be easily generalised for the multiple zeta functions.

Some of the natural variants of the multiple zeta functions are the multiple zeta-star functions (MZSF), the multiple $t$-functions (M$t$F) and the multiple $t$-star functions (M$t$SF), and these functions are defined on $U_r$ by
$$
\zeta^{\star}_r(s_1,\ldots,s_r):=\sum_{n_1\geq\cdots\geq n_r\geq 1}\frac{1}{n_1^{s_1}\cdots n_r^{s_r}},
$$
$$
t_r(s_1,\ldots,s_r):=\sum_{n_1>\cdots> n_r>0}\frac{1}{(2n_1-1)^{s_1}\cdots (2n_r-1)^{s_r}},
$$
and 
$$
t^{\star}_r(s_1,\ldots,s_r):=\sum_{n_1\geq\cdots\geq n_r\geq 1}\frac{1}{(2n_1-1)^{s_1}\cdots (2n_r-1)^{s_r}},
$$
respectively, for any integer $r \ge 1$. As before, in the depth $0$ case, we set these functions to be $1$. The MZSF and M$t$SF satisfy similar harmonic product formula. For example, 
\begin{align}\label{eq-stuffle-star}
\zeta^{\star}(s_1)\zeta^{\star}(s_2)=\zeta^{\star}_2(s_1,s_2)+\zeta^{\star}_2(s_2,s_1)-\zeta^{\star}(s_1+s_2), \text{ for $\Re(s_1), \Re(s_2)>1$}.
\end{align}

The meromorphic continuation of these functions to $\C^r$ have been studied by many authors. A particular case, which is of our interest in this article, is obtained by considering these meromorphic functions at the points of the form $(s,\ldots,s)$ with $s\in \R$. We use the notation $f_r(s)$ to denote the function $f_r(s,\ldots,s)$, where $f$ is any one of $\zeta,\zeta^\star,t,t^\star$. Hoffman \cite{H}, \cite{H2} computed some special values of these functions at positive integers. For example,
\begin{align}\label{at-2}
\zeta_r(2)=\zeta_r(2,\ldots,2)=\frac{\pi^{2r}}{(2r+1)!}, \quad
t_r(2)=t_r(2,\ldots,2)=\frac{\pi^{2r}}{2^{2r}(2r)!}, 
\end{align}
and 
\begin{align}\label{at-22}
t^{\star}_r(2)=t^{\star}_r(2,\ldots,2)=\frac{\pi^{2r}}{2^{2r}(2r)!}(-1)^nE_{2n},
\end{align}
where $E_n$ is the $n$-th Eulerian number. Muneta \cite{M} also obtained some special values of the multiple zeta-star function at positive integers. A more systematic study of the multiple zeta functions and the multiple Hurwitz zeta functions at identical arguments was initiated by Kamano \cite{K}. He computed the special values of these functions at non-positive integers by using the harmonic product.



One can view $\zeta_r(s)$ and its variants in terms of the symmetric functions in infinite number of variables. To see this, for an integer $r\geq 1$, let $e_r$ denote the elementary symmetric function and $h_r$ denote the complete symmetric function in $n$ variable, defined by
$$
e_r=e_r(x_1,x_2,\ldots,x_n):=\sum_{n\geq n_1>n_2>\cdots>n_r>0}x_{n_1}x_{n_2}\cdots x_{n_r},
$$
and 
$$
h_r=h_r(x_1,x_2,\ldots,x_n):=\sum_{n\geq n_1\geq n_2\geq \cdots \geq n_r\geq 1 }x_{n_1}x_{n_2}\cdots x_{n_r}.
$$
Then for $\Re(s)>1$, we have
$$
\zeta_r(s)=\lim_{n\rightarrow\infty}e_r\left(\frac{1}{1^s},\frac{1}{2^s},\ldots,\frac{1}{n^s}\right), \quad 
\zeta^{\star}_r(s)=\lim_{n\rightarrow\infty}h_r\left(\frac{1}{1^s},\frac{1}{2^s},\ldots,\frac{1}{n^s}\right),
$$
and 
$$
t_r(s)=\lim_{n\rightarrow\infty}e_r\left(\frac{1}{1^s},\frac{1}{3^s},\ldots,\frac{1}{(2n-1)^s}\right), \quad 
t^{\star}_r(s)=\lim_{n\rightarrow\infty}h_r\left(\frac{1}{1^s},\frac{1}{3^s},\ldots,\frac{1}{(2n-1)^s}\right).
$$
For an integer $r \ge 1$, if $p_r$ denote the symmetric function of $r$-th power sum in $n$ variables, defined by  
$$
p_r=p_r(x_1,x_2,\ldots,x_n):=\sum_{i=1}^{n}x_i^r,
$$
then we have the following well-known Newton's identities:
\begin{align*}
re_r=\sum_{j=1}^{r}(-1)^{j-1}e_{r-j}p_j
\ \text{ and } \
rh_r=\sum_{j=1}^{r}h_{r-j}p_j.
\end{align*}
These identities give rise to the following recursive formulae
\begin{equation}\label{fundamental-formula}
r\zeta_r(s)=\sum_{j=1}^{r}(-1)^{j-1}\zeta_{r-j}(s)\zeta(js), \quad rt_r(s)=\sum_{j=1}^{r}(-1)^{j-1}t_{r-j}(s)t(js),
\end{equation}
and
\begin{equation}\label{fundamental-formula-star}
r\zeta^{\star}_r(s)=\sum_{j=1}^{r}\zeta^{\star}_{r-j}(s)\zeta(js), \quad rt^{\star}_r(s)=\sum_{j=1}^{r}t^{\star}_{r-j}(s)t(js).
\end{equation}
Note that these identities are also consequences of \eqref{eq-Euler} and \eqref{eq-stuffle-star}. With these recurrence formulae, we can write $\zeta_r(s)$ and its variants, in terms of the usual Riemann zeta function. For example,
\begin{equation}\label{zeta_2-zeta_3}
\zeta_2(s)=\frac{1}{2}\left\{\zeta(s)^2-\zeta(2s)\right\}, \quad \zeta_3(s)=\frac{1}{6}\left\{\zeta(s)^3-\zeta(s)\zeta(2s)+2\zeta(3s)\right\},
\end{equation}
and 
\begin{equation}\label{zeta^{star}_2-zeta^{star}_3}
\zeta^{\star}_2(s)=\frac{1}{2}\left\{\zeta(s)^2+\zeta(2s)\right\}, \quad \zeta^{\star}_3(s)=\frac{1}{6}\left\{\zeta(s)^3+3\zeta(s)\zeta(2s)+2\zeta(3s)\right\}.
\end{equation}
In \cite[eq. 2.13]{MMT}, Matsumoto, Matsusaka and Tanackov  wrote them in the following form:
\begin{align}\label{Matsumoto-zeta}
\zeta_r(s)=\sum_{\substack{{\bm{\alpha}}=(\alpha_1, \ldots,  \alpha_q)\in \N^q\\ {\bm{\beta}}=(\beta_1, \ldots,  \beta_q)\in \N^q \\ q\leq r, \ \alpha_1\beta_1+\cdots+\alpha_q\beta_q=r}}C({\bm \alpha}, {\bm \beta})\prod_{j=1}^{q}\zeta(\alpha_j s)^{\beta_j},
\end{align}
where $C({\bm \alpha}, {\bm \beta})$ are some rational numbers, which were not expressed explicitly though. Note that knowing these constants explicitly gives a complete description of the special values of $\zeta_r(s)$ in terms of the special values of $\zeta(s)$. Firstly, we note that an explicit description of these coefficients can be given in terms of the values of the complete Bell polynomials.

\begin{thm}\label{main-thm}
Let $r\geq 1$ be an integer and $s\in \R$ with $s>1$. Then
\begin{align}\label{r-fold-zeta}
\zeta_r(s)=\sum_{ \substack{c_1,c_2, \ldots,c_r \geq 0, \\c_1+2c_2+\cdots+rc_r=r}}\frac{(-1)^{r+c_1+c_2+\cdots+c_r}}{1^{c_1}c_1! \ 2^{c_2}c_2!\cdots r^{c_r}c_r!}\zeta(s)^{c_1}\zeta(2s)^{c_2}\cdots\zeta(rs)^{c_r},
\end{align}
and 
\begin{align}\label{r-fold-t}
t_r(s)=\sum_{ \substack{c_1,c_2, \ldots,c_r \geq 0, \\c_1+2c_2+\cdots+rc_r=r}}\frac{(-1)^{r+c_1+c_2+\cdots+c_r}}{1^{c_1}c_1! \ 2^{c_2}c_2!\cdots r^{c_r}c_r!}t(s)^{c_1}t(2s)^{c_2}\cdots t(rs)^{c_r}.
\end{align}
\end{thm}

\begin{rmk}\label{rmk-MMT}\rm
 In \cite[Theorem 4.2, 4.3]{MMT}, Matsumoto, Matsusaka and Tanackov studied the singularities of $\zeta_r(s)$ and showed that $\zeta_r(s)$ has only polar singularities at $s=1/k$ of order $[r/k]$ for any $1\leq k\leq r$, where $[x]$ denotes the largest integer $ \le x$. Moreover, they also deduced the leading coefficient in the Laurent expansion of $\zeta_r(s)$ at these points. Using \eqref{r-fold-zeta}, we can immediately recover their result to show that if $r=kl+m$ with $0 \le m <k$, then 
 $
 \zeta_r(s) \sim \frac{(-1)^{(k+1)l}}{k^l l!}\zeta_m\left(\frac{1}{k}\right) (sk-1)^{-l},
 $
 around $s=1/k$. A short proof is given in Section \ref{sec-proof}.
\end{rmk}

Using the recurrence relation \eqref{fundamental-formula-star}, we also prove the following theorem for the star variants.

\begin{thm}\label{main-thm-2}
Let $r\geq 1$ be an integer and $s\in \R$ with $s>1$. Then
\begin{align}\label{r-fold-zeta-star}
\zeta^{\star}_r(s)=\sum_{ \substack{c_1,c_2, \ldots,c_r \geq 0, \\c_1+2c_2+\cdots+rc_r=r}}\frac{1}{1^{c_1}c_1! \ 2^{c_2}c_2!\cdots r^{c_r}c_r!}\zeta(s)^{c_1}\zeta(2s)^{c_2}\cdots \zeta(rs)^{c_r},
\end{align}
and 
\begin{align}\label{r-fold-t-star}
t^{\star}_r(s)=\sum_{ \substack{c_1,c_2, \ldots,c_r \geq 0, \\c_1+2c_2+\cdots+rc_r=r}}\frac{1}{1^{c_1}c_1! \ 2^{c_2}c_2!\cdots r^{c_r}c_r!}t(s)^{c_1}t(2s)^{c_2}\cdots t(rs)^{c_r}.
\end{align}
\end{thm}

The formulae \eqref{r-fold-t} and \eqref{r-fold-t-star} can be found in \cite[Corollary 3.1-3.2]{H2}, but novelty of our method is that we prove these results using the Bell polynomials. A brief discussion on the Bell polynomials is given in the next section. Our proof involving the Bell polynomials leads us to many other interesting consequences. For example, we get the following seemingly new relation among $\zeta_r(s)$, $\zeta_r^{\star}(s)$, $t_r(s)$ and $t_r^{\star}(s)$.

\begin{thm}\label{cor-int-rel} 
Let $r\geq 1$ be an integer.  Then for any $s\in \R$ with $s>1$, we have
$$
t_r(s)=\sum_{j=0}^{r}\frac{(-1)^{r-j}}{2^{(r-j)s}}\zeta_j(s)\zeta^{\star}_{r-j}(s) \ \text{ and }\
t^{\star}_r(s)=\sum_{j=0}^{r}\frac{(-1)^{r-j}}{2^{(r-j)s}}\zeta_{r-j}(s)\zeta^{\star}_{r}(s).
$$
\end{thm}

\begin{rmk}\rm
The left hand side and the right hand side of each of \eqref{r-fold-zeta},  \eqref{r-fold-t}, \eqref{r-fold-zeta-star} and \eqref{r-fold-t-star} have meromorphic extensions. Therefore, these identities can be considered as identities of meromorphic functions on $\C$ and can be used to study $\zeta_r(s)$, $t_r(s)$ and their star-variants at non-positive integer points. 
\end{rmk}


One can easily see that  $\zeta_r(2k),t_r(2k),\zeta_r^{\star}(2k),t_r^{\star}(2k) \in \Q^* \pi^{2k}$ and $\zeta_r(-2k)=t_r(-2k)=\zeta_r^{\star}(-2k)=t_r^{\star}(-2k)=0$ for all $k\geq 1$. But here, we also look at the value at $0$.  We prove that $\zeta_r(0)$ and $\zeta^{\star}_r(0)$ are non-zero rational numbers, whereas $t_r(0)=t^{\star}_r(0)=0$.
As a corollary to the proof of \thmref{main-thm} and \thmref{main-thm-2}, we obtain the exact values in terms of certain binomial coefficients. Many such interesting results are included in Section \ref{sec-app}.


We conclude this section by mentioning that Merca \cite{MM} obtained the following similar looking expression for $\zeta(2k)$, for an integer $k\geq 1$:
\begin{align*}
\zeta(2k)&=k\pi^{2k}\sum_{ \substack{c_1,c_2, \ldots,c_k \geq 0, \\ c_1+2c_2+\cdots+kc_k=k}}\frac{(-1)^{k+c_1+\cdots+c_k}(c_1+\cdots+c_k-1)!}{3!^{c_1}c_1! \ 5!^{c_2}c_2!\cdots (2k+1)!^{c_k}c_k!}\\
&=\frac{(2\pi)^{2k}}{2(2^{2k}-2)}\sum_{\substack{c_1,c_2,\ldots,c_k\geq 0\\ c_1+2c_2+\cdots+kc_k=k}}\frac{(-1)^{k+c_1+\cdots+c_k}(c_1+\cdots+c_k)!}{3!^{c_1}c_1! \ 5!^{c_2}c_2!\cdots (2k+1)!^{c_k}c_k!}.
\end{align*}


\section{Preliminaries}

We first recall the definition of the Bell polynomials. Let $n\geq k\geq 1$ be integers. The partial Bell polynomial   $\mathbf{B}_{n,k}:=\mathbf{B}_{n,k}(x_1,\ldots,x_{n-k+1})$ is a polynomial in (infinitely many) indeterminates  $x_1, x_2,\ldots$, defined by the coefficients of the formal double series
\begin{align*}
\phi(t,u):=\exp\left(u\sum_{m\geq1}x_m\frac{t^m}{m!}\right)=1+\sum_{n\geq k\geq 1}\mathbf{B}_{n,k}\frac{t^n}{n!}u^k
=1+\sum_{n\geq 1}\left(\sum_{k=1}^{n}u^k\mathbf{B}_{n,k}\right)\frac{t^n}{n!}.
\end{align*}
It is natural to set ${\bf B}_{0,0}=1$. Note that by the above definition, $\mathbf{B}_{n,k}=0$ for all $1\leq n<k$.  It is easy to see that
$$
{\bf B}_{1,1}=x_1, {\bf B}_{2,1}=x_2, {\bf B}_{2,2}=x_1^2, {\bf B}_{3,1}=x_3, {\bf B}_{3,2}=3x_1x_2, {\bf B}_{3,3}=x_1^3.
$$
In general, one can see that for $n \ge 1$,
$$
{\bf B}_{n,1}=x_n, {\bf B}_{n,n}=x_1^{n}.
$$
The partial Bell polynomials satisfy many recurrence formulae and special values of these functions are well-studied (for more details see \cite{L}).
One often assigns the weight $i$, to the indeterminate $x_i$ and then the weight of a monomial is defined to be the sum of the weight of each of the indeterminate present in that monomial (counted with multiplicity). We now recall the following key theorem of Bell \cite{B}, after whom these polynomials are named. 

\begin{thm}\label{them-Bell}
For integers $n\geq k\geq 1$, the partial Bell polynomial $\mathbf{B}_{n,k}$ is a homogeneous polynomial of degree $k$ and  weight $n$, and has integer coefficients. More precisely,
$$
\mathbf{B}_{n,k}(x_1, x_2, \ldots, x_{n-k+1})=\sum_{ \substack{c_1,c_2, \ldots,c_n \geq 0, \\c_1+2c_2+3c_3+\cdots+nc_n=n,\\ c_1+c_2+\cdots+c_n=k}}\frac{n!}{(1!)^{c_1}c_1!\  (2!)^{c_2}c_2!\cdots (n!)^{c_n}c_n!}x_1^{c_1}x_2^{c_2}\cdots x_n^{c_n}.
$$
\end{thm}

One of the many remarkable application of these Bell polynomials is that they can be used to express Fa\`{a} di Bruno's formula for the derivative of the composition of two functions (see \cite{WJ}). This then leads to many interesting applications of Fa\`{a} di Bruno's formula (see also \cite{TM}).

For an integer $n\geq 1$, the complete Bell polynomial $\mathbf{Y}_n(x_1,\ldots, x_n)$ is defined by 
$$
\mathbf{Y}_n(x_1,\ldots, x_n)=\sum_{k=1}^{n}\mathbf{B}_{n,k}(x_1,\ldots, x_{n-k+1}).
$$ 
For example,
$$
\mathbf{Y}_1(x_1)=x_1, \ \mathbf{Y}_2(x_1,x_2)=x_1^2+x_2, \ \mathbf{Y}_3(x_1,x_2,x_3)=x_1^3+3x_1x_2+x_3,
$$
and in general, using \thmref{them-Bell}, we have
\begin{align}\label{eq-c-bell}
\mathbf{Y}_n(x_1,\ldots,x_n)=\sum_{ \substack{c_1,c_2, \ldots,c_n \geq 0, \\c_1+2c_2+3c_3+\cdots+nc_n=n}}\frac{n!}{(1!)^{c_1}c_1! \ (2!)^{c_2}c_2!\cdots (n!)^{c_n}c_n!}{x_1}^{c_1}{x_2}^{c_2}\cdots {x_n}^{c_n}.
\end{align}
It is easy to see from this expression that for any  $\alpha\in \C$, 
\begin{align}\label{hom-prop}
\mathbf{Y}_{n}(\alpha x_1,\alpha^2x_2,\ldots,\alpha^nx_n)=\alpha^n\mathbf{Y}_{n}(x_1,x_2,\ldots,x_n).
\end{align}
From the definition of the partial Bell polynomials, we see that the generating function for the complete Bell polynomials is given by
\begin{align}\label{gen-complete-Bell}
\phi(t,1)=\exp\left(\sum_{m\geq 1}x_m\frac{t^m}{m!}\right)
=1+\sum_{n\geq 1}\mathbf{Y}_n(x_1,\ldots,x_n)\frac{t^n}{n!}.
\end{align}
It is natural to set ${\bf Y}_{0}=1$. Note that we also have 
\begin{align}\label{eq-sum-variable}
\mathbf{Y}_n(x_1+y_1,x_2+y_2,\ldots, x_n+y_n)=\sum_{j=0}^{n}\binom{n}{j}\mathbf{Y}_j(x_1,x_2\ldots, x_{j})\mathbf{Y}_{n-j}(y_1,y_2\ldots, y_{n-j}).
\end{align}

\section{Proof of the theorems}\label{sec-proof}
\begin{proof}[Proof of \thmref{main-thm}.]
Since $\zeta_r(s)$ and $t_r(s)$ satisfy the same recurrence formula \eqref{fundamental-formula}, we give the key steps of the proof of \eqref{r-fold-zeta}. Recall \eqref{fundamental-formula} for $\zeta_r(s)$, namely,
\begin{equation*}
r\zeta_r(s)=\sum_{j=1}^{r}(-1)^{j-1}\zeta_{r-j}(s)\zeta(js).
\end{equation*}
For a real number with $s>1$, assume that $a_r=(-1)^r\zeta_r(s)$ and $b_r=\zeta(rs)$ for all $r\geq 1$. If we set $a_0=1$, then we have 
\begin{equation}\label{rec-ai}
ra_r=-\sum_{j=1}^{r}a_{r-j}b_j.
\end{equation}
To prove the required result, it is enough to write $a_r$ in terms of $b_1, \ldots, b_r$ using the above recurrence formula. First, consider the formal power series
\begin{align*}
p(x):=\sum_{n\geq 0}a_nx^n, \quad \text{and}\quad
q(x):=\sum_{n\geq 1}b_nx^n.
\end{align*}
From recurrence formula \eqref{rec-ai}, we have 
\begin{align*}
\sum_{n\geq 1}na_nx^{n-1}=-\sum_{n\geq 1}\sum_{i=1}^{n}a_{n-i}b_ix^{n-1}=
-\sum_{n\geq 1}\sum_{i=1}^{n}a_{n-i}b_ix^{n-1}=-\sum_{i\geq 1}b_i\sum_{n\geq i}a_{n-i}x^{n-1}.
\end{align*}
Now,
\begin{align*}
\begin{split}
\sum_{n\geq 1}na_nx^{n-1}=-\sum_{i\geq 1}b_i\sum_{n\geq i}a_{n-i}x^{n-1}&=-\sum_{i\geq 1}b_i\sum_{n\geq 0}a_{n}x^{n+i-1}
=-g(x)\sum_{i\geq 1}b_ix^{i-1}
=-\frac{p(x)q(x)}{x}.
\end{split}
\end{align*}
Therefore,
$$
p'(x)=-\frac{p(x)q(x)}{x}.
$$
The solution of this differential equation is given by
\begin{align*}
p(x)=\exp\left(-\sum_{n\geq 1}\frac{b_n}{n}x^n\right).
\end{align*}
By comparing the above equation with \eqref{gen-complete-Bell}, we get for every integer $r\geq 1$,
\begin{align*}
a_r&=\frac{1}{r!}\mathbf{Y}_r\left(-0!b_1, -1!b_2, \ldots,-(r-1)!b_r\right).
\end{align*}
By substituting the values of $a_i$'s and $b_i$'s in the above equation we get
\begin{align}\label{eq-zeta}
\zeta_r(s)=\frac{(-1)^r}{r!}\mathbf{Y}_r\left(-0!\zeta(s), -1!\zeta(2s),\ldots,-(r-1)!\zeta(rs)\right).
\end{align}
Now using \eqref{eq-c-bell}, we get the desired expression \eqref{r-fold-zeta}. Similarly, taking $a_r=(-1)^rt_r(s)$ and $b_r=t(rs)$, we get the identity
\begin{align}\label{eq-t-r}
t_r(s)=\frac{(-1)^r}{r!}\mathbf{Y}_r\left(-0!t(s), -1!t(2s), \ldots,-(r-1)!t(rs)\right),
\end{align}
which gives \eqref{r-fold-t}, using \eqref{eq-c-bell}.
\end{proof}

\begin{rmk}\label{rmk-zeta-r}\rm
As $ \mathbf{Y}_2(x_1,x_2)=x_1^2+x_2, \ \mathbf{Y}_3(x_1,x_2,x_3)=x_1^3+3x_1x_2+x_3$, we get back \eqref{zeta_2-zeta_3} by taking $x_i=-(i-1)!\zeta(is)$ for $i=1, 2, 3$.
\end{rmk}

\begin{proof}[Sketch of the proof of \thmref{main-thm-2}] We recall \eqref{fundamental-formula-star}, namely,
\begin{align*}
r\zeta^{\star}_r(s)=\sum_{j=1}^{r}\zeta^{\star}_{r-j}(s)\zeta(js) \ \text{ and } \ rt^{\star}_r(s)=\sum_{j=1}^{r}t^{\star}_{r-j}(s)t(js).
\end{align*}
To prove identity \eqref{r-fold-zeta-star}, we take $a_r= \zeta^{\star}_{r}(s)$ and $b_r=\zeta(rs)$ for $r\geq 1$ and $a_0=1$. Then following the proof of \thmref{main-thm}, we get 
$$
p'(x)=\frac{p(x)q(x)}{x},
$$
where $p(x)=\sum_{n\geq 0}a_nx^n$ and $q(x)=\sum_{n\geq 1}b_nx^n$. This gives 
$$
p(x)=\exp\left(\sum_{n\geq 1}\frac{b_n}{n}x^n\right).
$$
Now again comparing the above equation with \eqref{gen-complete-Bell}, and substituting the values of $a_i$'s and $b_i$'s we get
\begin{align}\label{eq-zeta-star}
\zeta^{\star}_r(s)&=\frac{1}{r!}\mathbf{Y}_r\left(0!\zeta(s), 1!\zeta(2s), \ldots,(r-1)!\zeta(rs)\right).
\end{align}
Similarly, taking $a_r=t^{\star}_r(s)$ and $b_r=t(rs)$, we get 
\begin{align}\label{eq-t-star}
t^{\star}_r(s)&=\frac{1}{r!}\mathbf{Y}_r\left(0!t(s), 1!t(2s), \ldots,(r-1)!t(rs)\right).
\end{align}
Using \eqref{eq-c-bell}, we get the desired expressions \eqref{r-fold-zeta-star} and \eqref{r-fold-t-star}. This completes the proof.
\end{proof}

\begin{proof}[Proof of \thmref{cor-int-rel}] 
Let $s>1$ be a real number. It is easy to note that
\begin{align}\label{zeta-t}
t(s)=t^{\star}(s)=\left(1-{2^{-s}}\right)\zeta(s).
\end{align}
Now we use \eqref{zeta-t} and \eqref{eq-sum-variable}, together with \eqref{hom-prop}, in \eqref{eq-t-r} to write
\begin{align*}
t_r(s)&= \frac{(-1)^r}{r!}\mathbf{Y}_r\left(-0!t(s), -1!t(2s),\ldots,-(r-1)!t(rs)\right)\\
&=\frac{(-1)^r}{r!}\mathbf{Y}_r\left(-0!(1-2^{-s})\zeta(s), -1!(1-2^{-2s})\zeta(2s),\ldots,-(r-1)!(1-2^{-rs})\zeta(rs)\right)\\
&\begin{aligned}
=\frac{(-1)^r}{r!}
\sum_{j=0}^{r}\binom{r}{j}\frac{1}{2^{(r-j)s}}&\mathbf{Y}_j(-0!\zeta(s), -1!\zeta(2s),\ldots,-(j-1)!\zeta(js))\times \\
&\mathbf{Y}_{r-j}\left(0!{\zeta(s)}, 
1!{\zeta(2s)},\ldots,(r-j-1)!{\zeta(js)}\right).
\end{aligned}
\end{align*}
Now using \eqref{eq-zeta} and \eqref{eq-zeta-star},  we get that
$$
t_r(s)=\frac{(-1)^r}{r!}
\sum_{j=0}^{r}\binom{r}{j}\frac{{(-1)^jj!(r-j)!}}{2^{(r-j)s}}\zeta_j(s)\zeta^{\star}_{r-j}(s)=\sum_{j=0}^{r}\frac{(-1)^{r-j}}{2^{(r-j)s}}\zeta_j(s)\zeta^{\star}_{r-j}(s).
$$
To get the identity for $t^{\star}_r(s)$, we follow the similar steps as above, starting with \eqref{eq-t-star}.
\end{proof}

\begin{proof}[Proof of \rmkref{rmk-MMT}]
As $\zeta(s)$ has only a simple pole at $s=1$ with residue $1$, it is evident from \eqref{r-fold-zeta} that $\zeta_r(s)$ has only possible poles at $s=1/k$ for $1\leq k\leq r$. Now to compute the order of the pole at $s=1/k$, we first note that the largest possible exponent of $\zeta(ks)$ in \eqref{r-fold-zeta}. Note that the largest possible value of  $c_k$ satisfying the equation 
\begin{align}\label{eq-part}
c_1+2c_2+\cdots+rc_r=r,
\end{align}
where $c_i$ are non-negative integers for all $1\leq i\leq r$, is exactly $l=[r/k]$.

In order to complete the proof, we first compute the coefficient of $(ks-1)^{-l}$ in the Laurent series expansion of $\zeta_r(s)$ around $s=1/k$ and prove that it is non-zero. Let $R(k;\zeta_{r})$ be the coefficient of $(ks-1)^{-l}$. Note that the condition 
$c_1+2c_2+\cdots+rc_r=r$ now reduces to conditions $c_1+2c_2+\cdots+mc_m=m$ and $c_i=0$ for $m<i \le r, i \ne k$, as we have $c_k=l$. Hence we have 
\begin{align*}
R(k;\zeta_{r})&=
\sum_{\substack{c_1,c_2, \ldots,c_m \geq 0, \\c_1+2c_2+\cdots+mc_m=m}}\frac{(-1)^{kl+m+c_1+c_2+\cdots+c_m+l}}{1^{c_1}c_1! \  2^{c_2}c_2!\cdots m^{c_m}c_m!\cdot k^ll!}\zeta\left(\frac{1}{k}\right)^{c_1}\zeta\left(\frac{2}{k}\right)^{c_2}\cdots\zeta\left(\frac{m}{k}\right)^{c_m}\\
&=\frac{(-1)^{l(k+1)}}{k^ll!}\sum_{\substack{c_1,c_2, \ldots,c_m \geq 0, \\c_1+2c_2+\cdots+mc_m=m}}\frac{(-1)^{m+c_1+c_2+\cdots+c_m}}{1^{c_1}c_1! \  2^{c_2}c_2!\cdots m^{c_m}c_m!}\zeta\left(\frac{1}{k}\right)^{c_1}\zeta\left(\frac{2}{k}\right)^{c_2}\cdots\zeta\left(\frac{m}{k}\right)^{c_m}.
\end{align*}
Since $\zeta(s)<0$ for all $s\in (0,1)$, we have $\zeta(i/k)<0$ for all $1\leq i\leq m<k$. Hence
\begin{align*}
\sum_{\substack{c_1,c_2, \ldots,c_m \geq 0, \\c_1+2c_2+\cdots+mc_m=m}}\frac{(-1)^{m+c_1+c_2+\cdots+c_m}}{1^{c_1}c_1! \  2^{c_2}c_2!\cdots m^{c_m}c_m!}\zeta\left(\frac{1}{k}\right)^{c_1}\zeta\left(\frac{2}{k}\right)^{c_2}\cdots\zeta\left(\frac{m}{k}\right)^{c_m}
\end{align*}
is non-zero with sign $(-1)^m$ and is equal to $\zeta_m(1/k)$ by \eqref{r-fold-zeta}.
\end{proof}

\section{Applications to the special values of $\zeta_r(s)$ and its variants}\label{sec-app}
We have already seen that the formulae \eqref{r-fold-zeta},  \eqref{r-fold-t}, \eqref{r-fold-zeta-star} and \eqref{r-fold-t-star}, lead us to the fact that $\zeta_r(2k),t_r(2k),\zeta_r^{\star}(2k),t_r^{\star}(2k) \in \Q^* \pi^{2k}$ and $\zeta_r(-2k)=t_r(-2k)=\zeta_r^{\star}(-2k)=t_r^{\star}(-2k)=0$ for all $k\geq 1$, as we know that
$$\zeta(-2k)=0 \text{ and } \zeta(2k)=(-1)^{k+1}\frac{B_{2k}(2\pi)^{2k}}{2(2k)!},$$ where $B_n$ denotes the $n$-th Bernoulli number. 
Now using \eqref{eq-zeta} and its variants, it is possible to give a more explicit formula, in terms of the Bell polynomials. Substituting $s=2k$ in \eqref{eq-zeta}, we get
\begin{align*}
\zeta_{r}(2k)&=\frac{(-1)^r}{r!}{\bf Y}_r\left(-\zeta(2k), -1!\zeta(4k), \ldots,-(r-1)!\zeta(2rk)\right)\\
&=\frac{(-1)^r}{r!}{\bf Y}_r\left(-\frac{(-1)^{k+1}B_{2k}}{2(2k)!}(2\pi)^{2k},\ldots,-(r-1)!\frac{(-1)^{rk+1}B_{2rk}}{2(2rk)!}(2\pi)^{2rk}\right).
\end{align*}
Using \eqref{hom-prop}, we have
$$
\zeta_r(2k)=\frac{(-1)^{r(k+1)}}{r!}(2\pi)^{2rk}{\mathbf{Y}_r\left(0!\frac{B_{2k}}{2(2k)!}, 1!\frac{B_{4k}}{2(4k)!},\ldots,(r-1)!\frac{B_{2rk}}{2(2rk)!}\right)}.
$$
Similarly, using \eqref{eq-zeta-star}, we get
$$
\zeta^{\star}_r(2k)=\frac{(-1)^{rk}}{r!}(2\pi)^{2rk}{\mathbf{Y}_r\left(-0!\frac{B_{2k}}{2(2k)!}, -1!\frac{B_{4k}}{2(4k)!},\ldots,-(r-1)!\frac{B_{2rk}}{2(2rk)!}\right)}.
$$
\begin{rmk}\label{rmk-value-2}\rm
Using \cite[eq. (1), (2)]{J}, we have 
$$
\mathbf{Y}_r\left(0!\frac{B_{2}}{2(2)!}, 1!\frac{B_{4}}{2(4)!},\ldots,(r-1)!\frac{B_{2r}}{2(2r)!}\right)=\frac{r!}{2^{2r}(2r+1)!},
$$
and 
$$
\mathbf{Y}_r\left(-0!\frac{B_{2}}{2(2)!}, -1!\frac{B_{4}}{2(4)!},\ldots,-(r-1)!\frac{B_{2r}}{2(2r)!}\right)=\frac{r!(2^{1-2r}-1)B_{2r}}{(2r)!}.
$$
Therefore, we recover the well known formulae
$$
\zeta_r(2)=\frac{\pi^{2r}}{(2r+1)!} \ \text{ and } \ \zeta^{\star}_r(2)=\frac{(-1)^{r+1}(2^{2r}-2)B_{2r}\pi^{2r}}{(2r)!}.
$$
One can now use \thmref{cor-int-rel} to derive analogous formulae for $t_r(2)$ and $t^\star_r(2)$ as well.
\end{rmk}


We now turn our attention to the special values at $0$. Note that since $t(0)=0$, using \eqref{r-fold-t} and \eqref{r-fold-t-star}, we get $t^{\star}_r(0)=t_r(0)=0.$
We now consider the special values $\zeta_r(0)$ and $\zeta^{\star}_r(0)$.
\begin{cor}\label{cor-value-zero}
Let $r\geq 1$ be an integer. Then  
\begin{align}\label{eq-zero}
\zeta_r(0)=\frac{(-1)^r}{4^r}\binom{2r}{r}\ \text{ and } \
\zeta^{\star}_r(0)=-\frac{1}{r2^{2r-1}}\binom{2r-2}{r-1}.
\end{align}
\end{cor}

\begin{proof}
Using \eqref{eq-zeta}, we have
\begin{align*}
\zeta_{r}(0)=\frac{(-1)^r}{r!}{\mathbf{Y}_r\left(-0!\zeta(0), -1!\zeta(0),\ldots,-(r-1)!\zeta(0)\right)}
=\frac{(-1)^r}{r!}\mathbf{Y}_r\left(\frac{1}{2}, \frac{1!}{2}, \ldots, \frac{(r-1)!}{2}\right).
\end{align*}
By definition,
$$
\mathbf{Y}_r\left(\frac{0!}{2}, \frac{1!}{2}, \ldots, \frac{(r-1)!}{2}\right)=\sum_{k=1}^{r}\mathbf{B}_{r,k}\left(\frac{0!}{2}, \frac{1!}{2},\ldots, \frac{(r-1)!}{2}\right)=\sum_{k=1}^{r}\frac{\mathbf{B}_{r,k}\left({0!}, {1!},\ldots, {(r-1)!}\right)}{2^k}.
$$
It is well known that (see \cite[Chapter 5]{L})
$$
\mathbf{B}_{r,k}\left({0!}, {1!},\ldots, {(r-1)!}\right)=|s(r,k)|,
$$
where $|s(r,k)|$ is the unsigned Stirling number of the first kind given by rising factorial, i.e.,
\begin{align}\label{gen-stirling}
x\left(x+1\right)\cdots \left(x+r-1\right)=\sum_{k=1}^{r}{|s(r,k)|}{x^k}.
\end{align}
Therefore, we have
$$
\mathbf{Y}_r\left(\frac{0!}{2}, \frac{1!}{2}, \ldots, \frac{(r-1)!}{2}\right)=\sum_{k=1}^{r}\frac{|s(r,k)|}{2^k}=\prod_{k=0}^{r-1}\left(k+\frac{1}{2}\right),
$$
and
$$
\mathbf{Y}_r\left(-\frac{0!}{2}, -\frac{1!}{2}, \ldots, -\frac{(r-1)!}{2}\right)=\sum_{k=1}^{r}\frac{(-1)^k|s(r,k)|}{2^k}=\prod_{k=0}^{r-1}\left(k-\frac{1}{2}\right),
$$
Hence, we get
$$
\zeta_r(0)=\frac{(-1)^r}{r!}\prod_{k=0}^{r-1}\left(k+\frac{1}{2}\right), \quad \zeta^{\star}_r(0)=\frac{1}{r!}\prod_{k=0}^{r-1}\left(k-\frac{1}{2}\right).
$$
Simplifying the above identities, we get the desired formulae.
\end{proof}
\begin{rmk}\label{asym-at-zero} \rm 
Using \cite[Corollary 2]{K} also one can get the value of $\zeta_r(0)$. But here our proof only uses of the complete Bell polynomials. 
\end{rmk}
\begin{rmk}\rm  Using Stirling's formula
$r! \sim \sqrt{2\pi r} \left( r/e \right)^r$,
as $r\rightarrow\infty$, we have $\left(2r \atop r \right) \sim 4^r/\sqrt{\pi r}$, as $r\rightarrow\infty$. Hence we have, as $r\rightarrow\infty$,
$$
\zeta_r(0)\sim\frac{(-1)^r}{\sqrt{\pi r}}, \quad \zeta^\star_r(0)\sim-\frac{1}{2\sqrt{\pi r^3}}
$$
\end{rmk}

\noindent {\bf Acknowledgements}:
The research of the author is supported by PMRF (grant number: 1402688, cycle 10). The author has no competing interests to declare that are relevant to the content of this article. This manuscript has no associated data.



\begin{thebibliography}{100}





%
\bibitem{B}
E. T. Bell, {\it Exponential polynomials}, Ann. of Math. (2) {\bf 35} (1934), no. 2, 258–277.
\bibitem{L}
L. Comtet, {\it Advanced Combinatorics, The Art of Finite and Infinite Expansions}, D. Reildel Publishing Company 1974.

\bibitem{H}
M. Hoffman, {\it Multiple harmonic series}, Pacific J. Math {\bf 152} (1992), no. 2, 275–290.

\bibitem{H2}
M. Hoffman, {\it An odd variant of multiple zeta values},
Commun. Number Theory Phys. {\bf 13} (2019), no. 3, 529–567.


\bibitem{J}
L. Jiu, D. Y. H. Shi, 
{\it Moments and cumulants on identities for Bernoulli and Euler numbers}, 
Math. Rep. (Bucur.)  {\bf 24(74)} (2022), no. 4, 643–650.

\bibitem{WJ}
W. P. Johnson,
The curious history of Fa\`a di Bruno's formula,
Amer. Math. Monthly {\bf 109} (2002), no. 3, 217--234.

\bibitem{K}

K. Kamano, {\it The multiple Hurwitz zeta function and a generalisation of Lerch's formula}, Tokyo J. Math. {\bf 29} (2006), 61–71.




\bibitem{MMT}
K. Matsumoto, T. Matsusaka, I. Tanackov,
{\it On the behaviour of multiple zeta-functions with identical arguments on the real line}, 
J. Number Theory {\bf 239} (2022), 151–182.

\bibitem{TM}
T. Matsusaka, 
Applications of Fa\`a di Bruno's formula to partition traces,
Res. Number Theory {\bf 11} (2025), no. 3, Paper No. 69, 11 pp.
 
\bibitem{MM}
M. Merca, {\it The Riemann Zeta Function With Even Arguments as Sums Over Integer Partitions}, The Amer. Math. Monthly, {\bf 124} (2017), no. 6, 554–557. 

\bibitem{M}
S. Muneta
{\it On some explicit evaluations of multiple zeta-star values},
J. Number Theory, {\bf 128} (2008), no. 9, pp. 2538–2548.


\end{thebibliography}
\end{document}